\documentclass[11pt,a4paper]{article}

\usepackage[flushmargin,hang,bottom]{footmisc} 			%Aufhebung von Einrückungen bei Fussnoten (muss vor Paket hyperref eingefügt werden)

\usepackage[
  bookmarks=true,
  bookmarksopen=true,
  bookmarksnumbered=true,
  pdfauthor={Benjamin Miesch},
  pdftitle={Injective Metrics on Cube Complexes},
  pdfsubject={Injective Cube Complexes},
  pdfstartview=FitH,
	colorlinks=true,
	linkcolor=blue
]{hyperref}

\usepackage[pdftex]{graphicx}	

\usepackage{caption}						%Ermöglicht Beeinflussung von Beschriftung von Abbildungen und Tabellen
\usepackage[T1]{fontenc}										%Schriftpaket
\usepackage[english]{babel} 								%Übersetzung sprachspezifischer Befehle (chapter -> Kapitel, table of contents -> Inhaltsverzeichnis, ...) 
\usepackage[ansinew]{inputenc} 							%Ermöglicht Umlaute im Tex-File 

\usepackage{amsmath,amsfonts,amssymb}				%Mathematikpaket
\usepackage{a4} 															%Bewirkt, dass Seitenbreite mehr ausgefüllt wird
\usepackage{enumerate}
\usepackage{pinlabel}

\usepackage{amsthm}

\theoremstyle{plain}
\newtheorem{thm}{Theorem}[section]
\newtheorem{prop}[thm]{Proposition}
\newtheorem{cor}[thm]{Corollary}
\newtheorem{lem}[thm]{Lemma}

\theoremstyle{definition}
\newtheorem{definition}[thm]{Definition}

\theoremstyle{remark}
\newtheorem*{rmk}{Remark}
\newtheorem*{cl}{Claim}

\newtheorem{case}{Case}

\begin{document}

\title{\textsc{Gluing hyperconvex metric spaces}}
\author{Benjamin Miesch}
\date{\vspace{3ex}}
\maketitle

\begin{abstract}
	We investigate how to glue hyperconvex (or injective) metric spaces such that the resulting space remains hyperconvex. We give two new criteria, saying that on the one hand gluing along strongly convex subsets and on the other hand gluing along externally hyperconvex subsets leads to hyperconvex spaces. Furthermore, we show by an example that these two cases where gluing works are opposed and cannot be combined.
\end{abstract}

%%%%%%%%%%%%%%%%%%%%%%%%%%%%%%%%%%%%%%%%%%%%%%%%%%%%%%%%%%%%%%%%%%%%%%%%%%%
% Verzeichnisse
%%%%%%%%%%%%%%%%%%%%%%%%%%%%%%%%%%%%%%%%%%%%%%%%%%%%%%%%%%%%%%%%%%%%%%%%%%%

%Inhaltsverzeichnis
%\tableofcontents

%Abbildungsverzeichnis
%\listoffigures

%Tabellenverzeichnis
%\listoftables

%%%%%%%%%%%%%%%%%%%%%%%%%%%%%%%%%%%%%%%%%%%%%%%%%%%%%%%%%%%%%%%%%%%%%%%%%%%
% Hauptteil 
%%%%%%%%%%%%%%%%%%%%%%%%%%%%%%%%%%%%%%%%%%%%%%%%%%%%%%%%%%%%%%%%%%%%%%%%%%%

%%%%%%%%%%%%%%%%%%%%%%%%%%%%%%%%%%%%%%%%%%%%%%%%%%%%%%%%%%%%%%%%%%%%%%%%%%%

\section{Introduction}

%%%%%%%%%%%%%%%%%%%%%%%%%%%%%%%%%%%%%%%%%%%%%%%%%%%%%%%%%%%%%%%%%%%%%%%%%%%

A metric space $(X,d)$ is called \emph{hyperconvex} if every collection of closed balls $\{B(x_i,r_i)\}_{i\in I}$ with $d(x_i,x_j) \leq r_i + r_j$ has non-empty intersection $\bigcap_i B(x_i,r_i) \neq \emptyset$. Hyperconvex spaces were introduced by Aronszajn and Panitchpakdi \cite{aronszajn} who proved that they are the same as injective metric spaces. Hyperconvex spaces play a crucial role in metric fixed point theory, see \cite{espinola} and the references therein.

A classical problem that arises in metric geometry is how to glue metric spaces such that their properties are preserved. Some attempts to solve this question for hyperconvex spaces can be found in \cite{piatek}, where it is shown that gluing along unique intervals preserves hyperconvexity.

As we will show, this result can be generalized to strongly convex subsets. A subset $A$ of a metric space $(X,d)$ is called \emph{strongly convex} if for each pair $x,y \in A$ the metric interval $I(x,y)= \{ z \in X : d(x,z)+d(z,y)=d(x,y) \}$ is contained in $A$. In section~\ref{sec:strongly_convex} we prove the following theorem.

\begin{thm}\label{thm:strongly_convex}
	Let $(X,d)$ be the metric space obtained by gluing a collection $(X_\lambda,d_\lambda)_{\lambda \in \Lambda}$ of hyperconvex metric spaces along some space $A$ such that $A$ is closed and strongly convex in $X_\lambda$ for each $\lambda \in \Lambda$. Then $(X,d)$ is hyperconvex as well.
\end{thm}

%In $l_\infty^n$ strongly convex subsets are all subsets obtained by cutting along diagonal sets, i.e. intersections of sets of the form $A = \{x\in l_\infty^n : \sum_{i=1}^n \epsilon_{i} x_i \geq a \}$ for $\epsilon_{i} \in \{ \pm 1 \}, a \in \mathbb{R}$.

Interesting objects that can be obtained by gluing are polyhedral or cubical complexes. One attempt to put a hyperconvex metric on cube complexes is to take a metric on each cube such that they are isometric to the unit cube $[0,1]^n$ in $l_\infty^n$ (see \cite{mai,miesch}). As we see there the gluing of two cubes along some face preserves hyperconvexity. But for $n\geq 2$ proper faces of $n$-cubes are far from being strongly convex. The aim of our second criterion is to characterize such sets in arbitrary hyperconvex spaces and show that gluing along these subsets preserves hyperconvexity in general.

A subset $A$ of a metric space $X$ is called \emph{externally hyperconvex} (cf. \cite{aronszajn}) if for any collection of balls $\{B(x_i,r_i)\}_{i \in I}$ in $X$ with $d(x_i,x_j) \leq r_i+r_j$ and $d(x_i,A)\leq r_i$ we have $A\cap \bigcap_i B(x_i,r_i) \neq \emptyset$. In section~\ref{sec:externally_hyperconvex} we first show that bounded externally hyperconvex subsets share some basic properties of balls in hyperconvex spaces.

\begin{prop}\label{prop:intersection}
	Let $(X,d)$ be a hyperconvex space and $\{A_i\}_{i \in I}$ a family of pairwise intersecting externally hyperconvex subsets such that one of them is bounded. Then $\bigcap_{i \in I} A_i \neq \emptyset$.
\end{prop}

Using this crucial property we prove that gluing along externally hyperconvex subsets preserves hyperconvexity. This answers a question raised in \cite{piatek}.

\begin{thm}\label{thm:externally_hyperconvex}
	Let $(X,d)$ be the metric space obtained by gluing a family of hyperconvex metric spaces $(X_\lambda,d_\lambda)_{\lambda \in \Lambda}$ along some set $A$ such that $A$ is externally hyperconvex in each $X_\lambda$. Then $X$ is hyperconvex.
	Moreover $A$ is externally hyperconvex in $X$.
\end{thm}

Finally we give an example showing that gluing two spaces $X_1,X_2$ along some subset $A$ such that $A \subset X_1$ is strongly convex and $A \subset X_2$ is externally hyperconvex does not preserve hyperconvexity in general.

%%%%%%%%%%%%%%%%%%%%%%%%%%%%%%%%%%%%%%%%%%%%%%%%%%%%%%%%%%%%%%%%%%%%%%%%%%%

\section{Preliminaries}

%%%%%%%%%%%%%%%%%%%%%%%%%%%%%%%%%%%%%%%%%%%%%%%%%%%%%%%%%%%%%%%%%%%%%%%%%%%

First we fix some notation. Let $(X,d)$ be a metric space. We denote by
\begin{equation*}
	B(x_0,r)= \{x \in X : d(x,x_0) \leq r \}
\end{equation*}
the closed ball of radius $r$ with center in $x_0$. For any subset $A \subset X$ let
\begin{equation*}
	B(A,r)= \{x \in X : d(x,A):= \inf_{y \in A} d(x,y) \leq r \}
\end{equation*}
be the closed $r$-neighborhood of $A$.

\begin{definition}\label{def:gluing}
Let $(X_\lambda,d_\lambda)_{\lambda\in\Lambda}$ be a family of metric spaces with closed subspaces $A_\lambda\subset X_\lambda$. Suppose that all $A_\lambda$ are isometric to some metric space $A$. For every $\lambda\in\Lambda$ fix some isometry $\varphi_\lambda \colon A \to A_\lambda$. We define an equivalence relation on the disjoint union $\bigsqcup_\lambda X_\lambda$ generated by $\varphi_\lambda(a) \sim \varphi_{\lambda'}(a)$ for $a \in A$. The resulting space $X = (\bigsqcup_\lambda X_\lambda)/\sim$ is called the \emph{gluing} of the $X_\lambda$ along $A$.
\end{definition}

$X$ admits a natural metric. For $x \in X_\lambda$ and $y\in X_{\lambda'}$ it is given by
	\begin{equation}
		d(x,y) = \begin{cases}
				d_\lambda(x,y), &\text{ if } \lambda=\lambda', \\
				\inf_{a \in A} \{ d_\lambda(x,\varphi_\lambda(a)) + d_{\lambda'}(\varphi_{\lambda'}(a),y) \}, &\text{ if } \lambda \neq \lambda'.
				\end{cases}
	\end{equation}
For details see for instance Lemma I.5.24 in \cite{bridson}.

Hereinafter, if there is no ambiguity, indices for $d_\lambda$ are dropped and the sets $A_\lambda=\varphi_\lambda(A) \subset X_\lambda$ are identified with $A$.

\begin{definition}\label{def:injective}
A metric space $(X,d)$ is \emph{injective} if for every isometric embedding $\iota \colon A \hookrightarrow Y$ of metric spaces and every 1-lipschitz map $f \colon A \to X$ there is some 1-lipschitz map $\bar f \colon Y \to X$ such that $f = \bar f \circ \iota$.
\end{definition}

Injective metric spaces are complete, geodesic and contractible. Moreover they have the following intersection property.

\begin{definition}\label{def:hyperconvex}
	A metric space $(X,d)$ is \emph{($m$-)hyperconvex} if for any collection $\{ B(x_i,r_i)\}_{i\in J}$ of closed balls with $d(x_i,x_j) \leq r_i + r_j$ (and $|J| \leq m$) we have 
	\begin{equation*}
		\bigcap_{i\in J} B(x_i,r_i) \neq \emptyset.
	\end{equation*}		
\end{definition}

\begin{prop}(Theorem 4 in \cite{aronszajn}).
	A metric space $(X,d)$ is injective if and only if it is hyperconvex.
\end{prop}

\begin{lem}\label{lem:multimedian}
	Let $(X,d)$ be a ($3$-)hyperconvex metric space. For $x,y,z \in X$ we have
	\begin{equation*}
		I(x,y)\cap I(y,z) \cap I(z,x) \neq \emptyset.
	\end{equation*}
\end{lem}

\begin{proof}
	Choose $\alpha, \beta, \gamma \geq 0$ such that
	\begin{align*}
		\alpha + \beta &= d(x,y), \\
		\alpha + \gamma &= d(x,z), \\
		\beta + \gamma &= d(y,z).
	\end{align*}
	Then $I(x,y)\cap I(y,z) \cap I(z,x)= B(x,\alpha) \cap B(y,\beta) \cap B(z,\gamma) \neq \emptyset$.
\end{proof}

\begin{definition}\label{def:proximinal}
	A subset $A$ of a metric space $(X,d)$ is \emph{proximinal} if for all $x\in X$ the intersection $B(x,d(x,A))\cap A$ is non-empty.
\end{definition}

\begin{rmk}
	Proximinal subsets are closed.
\end{rmk}

%%%%%%%%%%%%%%%%%%%%%%%%%%%%%%%%%%%%%%%%%%%%%%%%%%%%%%%%%%%%%%%%%%%%%%%%%%%

\section{Gluing along strongly convex subsets}\label{sec:strongly_convex}

%%%%%%%%%%%%%%%%%%%%%%%%%%%%%%%%%%%%%%%%%%%%%%%%%%%%%%%%%%%%%%%%%%%%%%%%%%%

\begin{definition}\label{def:gated}
	A subset $A$ of a metric space $(X,d)$ is \emph{gated} if for all $x \in X$ there is some $\bar{x} \in A$ such that for all $a\in A$ we have $d(x,a)=d(x,\bar{x})+d(\bar{x},a)$. Clearly if such an $\bar{x}$ exists it is unique and we then call $\bar{x}$ the \emph{gate} of $x$ in $A$.
\end{definition}

\begin{lem}\label{lem:gated}\emph{(cf. Lemma 1.82 in \cite{moezzi}).}
	Let $A$ be a subset of the ($3$-)hyperconvex metric space $(X,d)$. Then $A$ is strongly convex and closed if and only if it is gated.
\end{lem}

\begin{proof}
	First assume that $A$ is strongly convex and closed. Fix $x \in X$. Let $x_n$ be a sequence of points in $A$ with $d(x,x_n) \leq d(x,A)+\frac{1}{n}$. For $n,k \in \mathbb{N}$ take $m_{n,k} \in I(x,x_n)\cap I(x,x_k) \cap I(x_n,x_k)$. By strong convexity we get $m_{n,k} \in A$ and hence
\begin{equation*}
	d(x_n,m_{n,k}) = d(x,x_n)-d(x_n,m_{n,k}) \leq d(x,A) + \frac{1}{n} - d(x,A) = \frac{1}{n}.
\end{equation*}
	By interchanging $x_n$ and $x_k$ we also get $d(x_k,m_{n,k}) \leq \frac{1}{k}$. Therefore
\begin{equation*}
	d(x_n,x_k)=d(x_n,m_{n,k})+d(m_{n,k},x_k) \leq \frac{1}{n} + \frac{1}{k},
\end{equation*}
	i.e. $x_n$ is a Cauchy sequence and since $A$ is closed it converges to some $\bar{x} \in A$. Moreover we have $d(x,\bar{x})=d(x,A)$.
	 We claim that $\bar{x}$ is a gate for $x$ in $A$. Let $y \in A$. By Lemma~\ref{lem:multimedian} there is some $z \in I(x,\bar{x})\cap I(\bar{x},y) \cap I(y,x)$. By convexity we have $z\in A$ and therefore $d(x,z) \geq d(x,A)=d(x,\bar{x})$. Since $z\in I(x, \bar{x})$ this implies $z=\bar{x}$ and hence $\bar{x} \in I(x,y)$ as desired. 
	On the other hand, if $A$ is gated for all points $x,y \in A$ and $z \in I(x,y)$ we have $z=\bar{z}$ and hence $I(x,y) \subset A$. Moreover for all $x \in X$ we have $d(x,A)=d(x,\bar{x})$ and therefore $\bar{x} \in B(x,d(x,A)) \cap A$, i.e. $A$ is proximinal and therefore closed.
\end{proof}

\begin{lem}\label{lem:hyperconvex}
	Let $A$ be a gated subset of a ($m$-)hyperconvex metric space $(X,d)$. Then $(A,d)$ is ($m$-)hyperconvex as well.
\end{lem}

\begin{proof}
	Let $\{ B(x_i,r_i) \}_{i\in J}$ be a collection of closed balls in $X$ with $d(x_i,x_j) \leq r_i + r_j$ and centers in $A$. Since $X$ is hyperconvex there is some $z\in \bigcap_{i\in J} B(x_i,r_i)$. Let $\bar{z}$ be the gate of $z$ in $A$. Then $\bar{z} \in \bigcap_{i\in J} B(x_i,r_i)\cap A$ since $d(x_i,\bar{z})=d(x_i,z)-d(z,\bar{z}) \leq r_i$ for all $i \in J$.
\end{proof}

\begin{lem}\label{lem:distance}
	Let $X$ be a metric space obtained by gluing the metric spaces $(X_\lambda,d_\lambda)_{\lambda\in\Lambda}$ along some set $A$. Assume that $A$ is gated in $X_\lambda$ for all $\lambda$. Then for $x\in X_\lambda, y\in X_{\lambda'}$ with $\lambda \neq \lambda'$ we have
	\begin{equation}
		d(x,y)=d(x,\bar{x}) + d(\bar{x},\bar{y}) + d(\bar{y},y).
	\end{equation}
\end{lem}

\begin{proof}
	For $a\in A$ we have
	\begin{align*}
		d(x,a)+d(a,y) &=  d(x,\bar{x}) + d(\bar{x},a) + d(a,\bar{y})+d(\bar{y},y) \\
			&\geq d(x,\bar{x}) + d(\bar{x},\bar{y})+d(\bar{y},y).
	\end{align*}		
\end{proof}

\begin{prop}\label{prop:strongly_convex}
	Let $(X,d)$ be the metric space obtained by gluing the collection $(X_\lambda,d_\lambda)_{\lambda \in \Lambda}$ of ($m$-)hyperconvex metric spaces along some space $A$ such that $A$ is closed and strongly convex in all $X_\lambda$. Then $(X,d)$ is ($m$-)hyperconvex as well.
\end{prop}

\begin{proof}
	Let $\{ B(x_i,r_i) \}_{i\in I}$ be a collection of closed balls in $X$ with $d(x_i,x_j) \leq r_i + r_j$. We need to show that $\bigcap_{i\in I} B(x_i,r_i)\neq \emptyset$. For any point $x \in X$ denote by $\bar{x}$ its gate in $A$. Moreover if $d(x_i,\bar{x_i}) \leq r_i$ define $\bar{r}_i=r_i-d(x_i,\bar{x_i})$. Observe that $B(\bar{x}_i,\bar{r}_i)\subset B(x_i,r_i)$. We distinguish three cases.
	
	\begin{case}
		First we assume $d(x_i,\bar{x}_i) \leq r_i$ and $B(\bar{x}_i,\bar{r}_i) \cap B(\bar{x}_j,\bar{r}_j) \neq \emptyset$ for all $i,j \in I$.
	\end{case}
	Since $A$ is hyperconvex we have $\bigcap_{i\in I} B(\bar{x}_i,\bar{r}_i) \neq \emptyset$ and hence 
	\begin{equation*}	
	\bigcap_{i\in I} B(x_i,r_i)\neq \emptyset.
	\end{equation*}

	\begin{case}
		Let again be $d(x_i,\bar{x}_i) \leq r_i$ for all $i\in I$ but assume that there are $i,j\in I$ such that 
		\begin{equation}\label{eq:empty}
			B(\bar{x}_i,\bar{r}_i) \cap B(\bar{x}_j,\bar{r}_j) = \emptyset,
		\end{equation}
		i.e. $d(x_i,\bar{x}_i)+d(\bar{x}_i,\bar{x}_j) + d(\bar{x}_j,x_j) > r_i+r_j$.
	\end{case}
	
	Observe that by Lemma~\ref{lem:distance} we have $B(\bar{x}_i,\bar{r}_i) \cap B(\bar{x}_j,\bar{r}_j) \neq \emptyset$ if $x\in X_\lambda$, $y\in X_{\lambda'}$ with $\lambda \neq \lambda'$ and therefore any two $x_i,x_j$ fulfilling \eqref{eq:empty} must be contained in some $X_\lambda$. We now claim that there is only one $X_{\lambda_0}$ containing such pairs.
	
	Let $x_1,x_2 \in X_{\lambda_0}$ be such that 
		\begin{equation*}
			B(\bar{x}_1,\bar{r}_1) \cap B(\bar{x}_2,\bar{r}_2) = \emptyset.
		\end{equation*}
	and $x_3,x_4 \in X_{\lambda}$ for some $\lambda \neq \lambda_0$. Define $r := \frac{d(\bar{x}_1,\bar{x}_2)-\bar{r}_1-\bar{r}_2}{2}$. We have 
	\begin{equation*}
		B(\bar{x}_1,\bar{r}_1+r) \cap B(\bar{x}_2,\bar{r}_2+r)\neq \emptyset
	\end{equation*}
	and therefore since $X_{\lambda}$ is hyperconvex there is some 
	\begin{equation*}	
	z\in B(\bar{x}_1,\bar{r}_1+r) \cap B(\bar{x}_2,\bar{r}_2+r)\cap B(x_3,r_3)\cap B(x_4,r_4).
	\end{equation*}
	But by the choice of $r$ we have $z\in I(x_1,x_2) \subset A$ and thus for $i=3,4$ we get
	\begin{equation*}
	d(\bar{x}_i,z)=d(x_i,z)-d(\bar{x}_i,x_i) \leq r_i-d(\bar{x}_i,x_i)=\bar{r}_i.
	\end{equation*}
	Hence we conclude $z \in B(\bar{x}_3,\bar{r}_3) \cap B(\bar{x}_4,\bar{r}_4) \neq \emptyset$.
	
	Denote $I_0 = \{ i \in I : x_i \in X_{\lambda_0} \}$. Then $\{ B(x_i,r_i) \}_{i\in I_0} \cup \{ B(\bar{x}_i,\bar{r}_i) \}_{i \in I \setminus I_0}$ is a family of pairwise intersecting balls in $X_{\lambda_0}$ and therefore has non-empty intersection $\bigcap_{i\in I_0} B(x_i,r_i) \cap \bigcap_{i \in I \setminus I_0} B(\bar{x}_i,\bar{r}_i)$ what implies $\bigcap_{i\in I} B(x_i,r_i) \neq \emptyset$.
	
	\begin{case}
		It remains the situation where $d(x_i,\bar{x}_i) > r_i$ for some $i\in J$.
	\end{case}	
	
	From Lemma~\ref{lem:distance} it immediately follows that all such $x_i$ are contained in some $X_{\lambda_0}$. Fix some $x_{i_0} \in X_{\lambda_0}$ with $d(x_{i_0},\bar{x}_{i_0}) > r_{i_0}$. Then for $x_i \notin X_{\lambda_0}$ we have $\bar{x}_{i_0} \in B(\bar{x}_i,\bar{r}_i)$ and therefore $B(\bar{x}_i,\bar{r}_i) \cap B(\bar{x}_{j},\bar{r}_{j})\neq \emptyset$ for $x_i,x_j \notin X_{\lambda_0}$. We conclude as above that $\bigcap_{i\in I} B(x_i,r_i) \neq \emptyset$.
\end{proof}

%%%%%%%%%%%%%%%%%%%%%%%%%%%%%%%%%%%%%%%%%%%%%%%%%%%%%%%%%%%%%%%%%%%%%%%%%%%

\section{Gluing along externally hyperconvex subspaces}\label{sec:externally_hyperconvex}

%%%%%%%%%%%%%%%%%%%%%%%%%%%%%%%%%%%%%%%%%%%%%%%%%%%%%%%%%%%%%%%%%%%%%%%%%%%

In the following let $(X,d)$ be a metric space.

\begin{definition}\label{def:admissible}
	A subset $A \subset X$ is called \emph{admissible} if it can be written as an intersection of closed balls $A= \bigcap_{i\in I} B(x_i,r_i)$. We denote the family of all admissible subsets of $X$ by $\mathcal{A}(X)$.
\end{definition}

\begin{rmk}
	A family of pairwise intersecting admissible sets in a hyperconvex space has non-empty intersection.
\end{rmk}

\begin{definition}\label{def:externally hyperconvex}
	A subset $A \subset X$ is called \emph{externally hyperconvex} if for all collections $\{B(x_i,r_i)\}_{i \in I}$ with $d(x_i,x_j) \leq r_i + r_j$ and $d(x_i,A) \leq r_i$ we have $\bigcap_i B(x_i,r_i) \cap A \neq \emptyset$. Denote the set of externally hyperconvex subsets of $X$ by $\mathcal{E}(X)$.
\end{definition}

\begin{rmk}
	Externally hyperconvex subsets are proximinal and therefore closed.
\end{rmk}

First we give the proof of two well known facts for externally hyperconvex subsets.

\begin{lem}\label{lem:admissible sets are externally hyperconvex}
	If $A\in \mathcal{A}(X)$ and $E \in \mathcal{E}(X)$ such that $A \cap E \neq \emptyset$ then $A \cap E \in \mathcal{E}(X)$. Especially if $X$ is hyperconvex we have $\mathcal{A}(X) \subset \mathcal{E}(X)$.
\end{lem}

\begin{proof}
	Since $A$ is admissible there is a collection of balls $\{B(x_i,r_i)\}_{i \in I}$ such that $A= \bigcap_i B(x_i,r_i)$. Now given a family of closed balls $\{B(x_j',r_j')\}_{j \in J}$ with $d(x_j',x_k') \leq r_j' + r_k'$ and $d(x_j',A\cap E) \leq r_j'$ we have $d(x_i,x_j') \leq r_i + r_j'$ and $d(x_i,E) \leq r_i$ and therefore $$A\cap E \cap \bigcap_j B(x_j',r_j') = E \cap \bigcap_i B(x_i,r_i) \cap \bigcap_j B(x_j',r_j') \neq \emptyset$$ since $E$ is externally hyperconvex.
	If $X$ is hyperconvex then $X \in \mathcal{E}(X)$ and therefore $\mathcal{A}(X) \subset \mathcal{E}(X)$.
\end{proof}

\begin{lem}\label{lem:externally hyperconvex neighborhood}
	Let $X$ be hyperconvex and $A \in \mathcal{E}(X)$. Then also $B(A,r) \in \mathcal{E}(X)$.
\end{lem}

\begin{proof}
	Let $\{B(x_i,r_i)\}_{i \in I}$ be a collection of closed balls with $d(x_i,x_j) \leq r_i + r_j$ and $d(x_i,B(A,r)) \leq r_i$. Then we also have $d(x_i,A)\leq r_i+r$. Since $A$ is externally hyperconvex there is some $y\in \bigcap B(x_i,r_i+r) \cap A$. Especially we have $d(x_i,y) \leq r_i + r$ and therefore since $X$ is hyperconvex we get $\emptyset \neq \bigcap B(x_i,r_i) \cap B(y,r) \subset \bigcap B(x_i,r_i) \cap B(A,r)$.
\end{proof}

The following technical lemma turns out to be the initial step in proving Proposition~\ref{prop:intersection}.

\begin{lem}\label{lem:key lemma}
	Let $X$ be a hyperconvex space. Let $A,A' \in \mathcal{E}(X)$ with $y \in A\cap A'\neq \emptyset$ and $x \in X$ with $d(x,A),d(x,A') \leq r$. Denote $d:=d(x,y)$ and $s:=d-r$. Then $A\cap A' \cap B(x,r) \cap B(y,s) \neq \emptyset$, given $s\geq 0$.
	In any case the intersection $A\cap A' \cap B(x,r)$ is non-empty. 
\end{lem}

\begin{proof}
	For $s \leq 0$ we have $y \in A \cap A' \cap B(x,r)$ Therefore let us assume $s > 0$.
	\begin{cl}
		For each $0 < l \leq s$ there are $a \in A, a'\in A'$ such that $d(a,a') \leq l$ and $a,a' \in B(x,r) \cap B(y,s)$.
	\end{cl}	
	We start choosing $$a_1 \in B(y,l) \cap B(x,d-l) \cap A$$ and $$a_1' \in B(y,l) \cap B(x,d-l) \cap B(a_1,l) \cap A'.$$ Then we inductively take $$a_n \in B(y,nl) \cap B(x,d-nl) \cap B(a_{n-1}',l) \cap A$$ and $$a_n' \in B(y,nl) \cap B(x,d-nl) \cap B(a_n,l) \cap A'$$ as long as $n \leq \lfloor \frac{s}{l} \rfloor =: n_0$. Finally there are $$a \in B(y,s) \cap B(x,r) \cap B(a_{n_0}',l) \cap A$$ and $$a' \in B(y,s) \cap B(x,r) \cap B(a,l) \cap A'$$ as desired.
	
	We now construct recursively two converging sequences $(a_n)_n\subset A$ and $(a_n')_n \subset A'$ such that $a_n,a_n' \in B(x,r) \cap B(y,s)$ with $$d(a_n,a_n') \leq \frac{1}{2^{n+1}} \text{ and } d(a_{n-1},a_n), d(a_{n-1}',a_n') \leq \frac{1}{2^n}.$$
	
	First choose $a_0,a_0' \in B(x,r) \cap B(y,s)$ with $d(a_0,a_0')\leq \frac{1}{2}$ according to the claim. Given $a_{n-1},a_{n-1}'$ with $d(a_{n-1},a_{n-1}') \leq \frac{1}{2^n}$ there is some $x_n \in B(a_{n-1},\frac{1}{2^{n+1}}) \cap B(a_{n-1}',\frac{1}{2^{n+1}}) \cap B(x,r-\frac{1}{2^{n+1}})$. Now applying the claim to $x_n$ and $y$ we find $$a_n,a_n' \in B(y,s) \cap B(x_n, \tfrac{1}{2^{n+1}}) \subset B(y,s) \cap B(x,r)$$ with $d(a_n,a_n') \leq \frac{1}{2^{n+1}}$. Moreover we have $$d(a_{n-1}, a_n) \leq d(a_{n-1},x_n) + d(x_n,a_n) \leq \frac{1}{2^{n+1}} + \frac{1}{2^{n+1}} = \frac{1}{2^n}.$$
	
	For $m\geq n$ we get $$d(a_n,a_m)\leq \sum_{k=n+1}^m d(a_{k-1},a_k)\leq \sum_{k=n+1}^\infty \frac{1}{2^k}=\frac{1}{2^n}$$ and similarly for $a_n'$. Hence the two sequences converge and since $d(a_n,a_n') \to 0$ they have a common limit point $a \in B(y,s)\cap B(x,r)\cap A \cap A'$.
\end{proof}

\begin{lem}\label{lem:intersection of three externally hyperconvex sets}
	Let $X$ be a hyperconvex space and let $A_0,A_1,A_2 \in \mathcal{E}(X)$ be pairwise intersecting externally hyperconvex subspaces. Then $A_0 \cap A_1 \cap A_2 \neq \emptyset$.
\end{lem}

\begin{proof}
	Choose some point $x_0 \in A_1 \cap A_2$ and let $r := d(x_0,A_0)$. By the previous lemma there is $y_0 \in A_0 \cap A_1 \cap B(x_0,r)$. Define $A_0' := A_0 \cap B(y_0,r) \in \mathcal{E}(X)$. Using again the lemma we have $A_0'\cap A_2 = A_0 \cap A_2 \cap B(y_0,r) \neq \emptyset$ and therefore there is some $z_0 \in A_0' \cap A_2 \cap B(x_0,r)=A_0 \cap A_2 \cap B(x_0,r)\cap B(y_0,r)$. Then since $A_0$ is externally hyperconvex, there is some $\bar{x}_0 \in B(x_0,r) \cap B(y_0, \frac{r}{2}) \cap B(z_0,\frac{r}{2})\cap A_0$ and using again Lemma~\ref{lem:key lemma}, we find $x_1 \in A_1 \cap A_2 \cap B(\bar{x}_0,\frac{r}{2})\cap B(x_0,\frac{r}{2})$. Proceeding this way we get some sequence $(x_n)_n \subset A_1 \cap A_2$ with $d(x_n, A_0) \leq \frac{r}{2^n}$ and $d(x_{n-1},x_n)\leq \frac{r}{2^n}$. Hence $d(x_n,x_m)\leq \sum_{k=n+1}^m \frac{r}{2^k} \leq \frac{r}{2^n}$ and therefore $(x_n)_n$ converges to $x \in A_0 \cap A_1 \cap A_2$.
\end{proof}

\begin{lem}\label{lem:intersection of externally hyperconvex is externally hyperconvex}
	If $A_0,A_1 \in \mathcal{E}(X)$ and $A_0 \cap A_1 \neq \emptyset$ then  $A_0 \cap A_1 \in \mathcal{E}(X)$.
\end{lem}

\begin{proof}
	Let $B(x_i,r_i)$ be a collection of balls with $d(x_i,x_j) \leq r_i + r_j$ and $d(x_i,A_1 \cap A_2) \leq r_i$. Define $A:= \bigcap B(x_i,r_i)$. Since the $A_k$ are externally hyperconvex we have $A \cap A_k = \bigcap_i B(x_i,r_i)\cap A_k \neq \emptyset$ and since admissible sets are externally hyperconvex we have $A_0 \cap A_1\cap A \neq \emptyset$ by Lemma~\ref{lem:intersection of three externally hyperconvex sets}.
\end{proof}

By induction we therefore get the following proposition.

\begin{prop}\label{prop:finite intersection}
	Let $A_0, \ldots, A_n \in \mathcal{E}(X)$ with $A_i\cap A_j \neq \emptyset$ then $\emptyset \neq \bigcap_{k=0}^n A_k \in \mathcal{E}(X)$.
\end{prop}

As a consequence of Baillon's theorem on the intersection of hyperconvex spaces \cite{baillon} the following theorem was proven by Esp\'inola and Khamsi in \cite{espinola}.

\begin{thm}\label{thm:decreasing intersection}
	\cite[Theorem 5.4]{espinola}. Let $\{A_i\}_{i \in I}$ be a descending chain of non-empty externally hyperconvex subsets of a bounded hyperconvex space $X$. Then $\bigcap_i A_i$ is non-empty and externally hyperconvex in $X$.
\end{thm}

Similarly to Corollary 8 in \cite{baillon} we can deduce the following corollary which implies Proposition~\ref{prop:intersection}.

\begin{cor}
	Let $\{A_i\}_{i \in I}$ be a family of pairwise intersecting externally hyperconvex subsets of a bounded hyperconvex space $X$. Then $\bigcap_i A_i$ is non-empty and externally hyperconvex in $X$.
\end{cor}

\begin{proof}
	Consider the set
	\begin{align*}
		\mathcal{F}= \left\{ J \subset I : \forall F \subset I \text{ finite,} \bigcap_{i \in J\cup F} A_i \neq \emptyset \text{ is externally hyperconvex}\right\}.
	\end{align*}
	By Proposition~\ref{prop:finite intersection} clearly $\emptyset \in \mathcal{F}$. Considering a chain $J_k \in \mathcal{F}$ and some finite set $F\subset I$, the sets $A_{J_k} = \bigcap_{i \in J_k \cup F} A_i$ build a decreasing chain of non-empty externally hyperconvex sets. Define $J=\bigcup_k J_k$. We have $A = \bigcap_{i \in J \cup F} A_i = \bigcap_k A_{J_k}$ is non-empty and externally hyperconvex by Theorem~\ref{thm:decreasing intersection}. Therefore $J \in \mathcal{F}$ is an upper bound of $J_k$. Hence $\mathcal{F}$ satisfies the hypothesis of Zorn's Lemma and therefore there is some maximal element $J_0 \in \mathcal{F}$. But for $i \in I$ we have $J_0 \cup \{i\} \in \mathcal{F}$ and by maximality of $J_0$ we conclude $I=J_0 \in \mathcal{F}$.
\end{proof}

%\begin{prop}\label{prop:countable intersection}
%	Let $X$ be a bounded hyperconvex space and $\{A_i\}_{i \in \mathbb{N}} \subset \mathcal{E}(X)$ a countable family of pairwise intersecting externally hyperconvex subsets. Then $\bigcap_{i \in \mathbb{N}} A_i$ is non-empty and externally hyperconvex in $X$.
%\end{prop}
%
%\begin{proof}
%	The sets $B_n = \bigcap_{i=1}^n A_i$ are a descending chain of nonempty externally hyperconvex subsets by Proposition~\ref{prop:finite intersection}. Therefore by Theorem~\ref{thm:countable intersection} $\bigcap_i A_i = \bigcap_n B_n \neq \emptyset$ is externally hyperconvex.
%\end{proof}

%\begin{proof}[Proof of Proposition~\ref{prop:intersection}.]
%	Let $A_{i_0}$ be bounded. Since $X$ is a separable metric space it is Lindel\"of. Therefore if $\bigcap_i A_i = \emptyset$ there is a countable subfamily $\{A_{i_n}\}_{n \in \mathbb{N}}$ with $\bigcap_{n \in \mathbb{N}} A_{i_n} = \emptyset$. But this also implies that $\{ A_{i_0} \cap A_{i_n} \}_{n\in \mathbb{N}}$ is a countable family of pairwise intersecting externally hyperconvex subsets of the bounded hyperconvex space $A_{i_0}$, contradicting Proposition~\ref{prop:countable intersection}.
%\end{proof}

\begin{prop}\label{prop:transitivity}
	Let $Y$ be an externally hyperconvex subset of the metric space $X$. Moreover let $A$ be externally hyperconvex in $Y$. Then $A$ is also externally hyperconvex in $X$.
\end{prop}

\begin{proof}
	Let $\{B(x_i,r_i)\}_{i \in I}$ be a collection of closed balls with $d(x_i,x_j) \leq r_i +r_j$ and $d(x_i,A) \leq r_i$. Then the sets $A_i:= B(x_i,r_i)\cap Y$ are externally hyperconvex subsets of $X$ and therefore also of $Y$. Clearly $A_i \cap A \neq \emptyset$ and since $Y$ is externally hyperconvex we have $A_i\cap A_j = B(x_i,r_i) \cap B(x_j,r_j) \cap Y \neq \emptyset$. Therefore we have a collection of pairwise intersecting externally hyperconvex subsets of $Y$ and hence by Proposition~\ref{prop:intersection} $A\cap \bigcap_i B(x_i,r_i) = A \cap \bigcap_i A_i \neq \emptyset$.
\end{proof}

Before proving Theorem~\ref{thm:externally_hyperconvex} we need some last technical lemmas. We use the convention that $B^\lambda(x,r)$ denotes the closed ball inside $X_\lambda$.

\begin{lem}\label{lem:exact distance}
	Let $X$ be a metric space obtained by gluing the hyperconvex spaces $X_\lambda$ along some set $A$ with $A \in \mathcal{E}(X_\lambda)$. Let $x \in X_\lambda$ and $y \in X_{\lambda'}$ with $\lambda \neq \lambda'$. Then for $s=d(x,A)$ there is some $a \in A\cap B^\lambda(x,s)$ such that $$d(x,y)=d(x,a)+d(a,y).$$	
\end{lem}

\begin{proof}
	Define $A'=A\cap B^\lambda(x,s) \neq \emptyset$. Observe that for $a \in A$ there is some $a' \in B^\lambda(x,s) \cap B^\lambda(a,d(a,x)-s) \cap A$ and hence $$d(x,a)+d(a,y) = d(x,a') + d(a',a) + d(a,y) \geq d(x,a')+d(a',y).$$ Therefore
	\begin{equation}
		d(x,y) = \inf_{a\in A'} d(x,a)+d(a,y) = s + d(A',y)
	\end{equation}
	By Lemma~\ref{lem:admissible sets are externally hyperconvex} and Proposition~\ref{prop:transitivity} we have $A' \in \mathcal{E}(X_{\lambda'})$. Thus there is some $a\in A'\cap B^{\lambda'}(y,d(A',y))$ and we get $d(x,y)=d(x,a)+d(a,y)$.
\end{proof}

\begin{lem}\label{lem:balls}
	Let $X$ be a metric space obtained by gluing the hyperconvex spaces $X_\lambda$ along some set $A$ with $A \in \mathcal{E}(X_\lambda)$. Let $x \in X_\lambda$ and $r \geq s:= d(x,A)$. Then for $\lambda_0 \neq \lambda$ we have 
	\begin{equation}
		B(x,r)\cap X_{\lambda_0} = B^{\lambda_0}(B^\lambda(x,s)\cap A,r-s).
	\end{equation}	
	Moreover $B(x,r)\cap X_{\lambda_0} \in \mathcal{E}(X_{\lambda_0})$.
\end{lem}

\begin{proof}
	Clearly $B(x,r)\cap X_{\lambda_0} \supset B^{\lambda_0}(B^\lambda(x,s)\cap A,r-s)$. Therefore assume $y \in B(x,r)\cap X_{\lambda_0}$. Then there is some $a\in A\cap B^\lambda(x,s)$ with $d(x,y) = d(x,a)+d(a,y)$ and $d(a,y) \leq r-s$ by Lemma~\ref{lem:exact distance}.
	
	We have $B^\lambda(x,s)\cap A \in \mathcal{E}(A)$ by Lemma~\ref{lem:admissible sets are externally hyperconvex} and therefore since $A \in \mathcal{E}(X_{\lambda_0})$ we also get $B^\lambda(x,s)\cap A \in \mathcal{E}(X_{\lambda_0})$ by Proposition~\ref{prop:transitivity}. Finally we conclude by Lemma~\ref{lem:externally hyperconvex neighborhood} that $B^{\lambda_0}(B^\lambda(x,s)\cap A,r-s) \in \mathcal{E}(X_{\lambda_0})$.
\end{proof}

\begin{proof}[Proof of Theorem~\ref{thm:externally_hyperconvex}.]
	Let $\{B(x_i,r_i)\}_{i\in I}$ be a collection of closed balls in $X$ with $d(x_i,x_j)\leq r_i+r_j$. First observe that there is at most one $\lambda_0 \in \Lambda$ such that $d(x_i,A) > r_i$ for some $x_i \in X_{\lambda_0}$. If there is none, fix any $\lambda_0 \in \Lambda$. Now define $A_i=B(x_i,r_i)\cap X_{\lambda_0}\neq \emptyset$. We claim that $A_i \cap A_j \neq \emptyset$ for all $i,j\in I$.
		
	Let $x_i \in X_\lambda$ and $x_j \in X_{\lambda'}$. First assume $\lambda,\lambda' \neq \lambda_0$. If $\lambda=\lambda'$ we have $A_i\cap A_j \neq \emptyset$ since $A \in \mathcal{E}(X_\lambda)$. If $\lambda \neq \lambda'$ there is some $a\in A\cap B(x_i,d(x_i,A))$ with $d(x_i,x_j)=d(x_i,a)+d(a,x_j)$ and therefore $\emptyset \neq B^{\lambda'}(a,r_i-d(x_i,A)) \cap B^{\lambda'}(x_j,r_j) \cap A \subset A_i \cap A_j$ by external hyperconvexity of $A$ in $X_{\lambda'}$. Finally if $\lambda'=\lambda_0$ we either get $B^{\lambda_0}(x_i,r_i) \cap B^{\lambda_0}(x_j,r_j) \neq \emptyset$ if $\lambda=\lambda'$ or $\emptyset \neq B^{\lambda_0}(a,r_i-d(x_i,A)) \cap B^{\lambda_0}(x_j,r_j) \subset A_i \cap A_j$ if $\lambda \neq \lambda'$ by hyperconvexity of $X_{\lambda_0}$. 
	
	Therefore using Lemma~\ref{lem:balls} the sets $A_i$ are a collection of bounded pairwise intersecting externally hyperconvex subsets of $X_{\lambda_0}$ and therefore have non-empty intersection $\bigcap_i A_i \neq \emptyset$ by Proposition~\ref{prop:intersection}, implying $\bigcap_i B(x_i,r_i) \neq \emptyset$.
	
	To see that $A$ is externally hyperconvex in $X$ proceed as before adding $A$ to the family $\{A_i\}_i$.
\end{proof}

%%%%%%%%%%%%%%%%%%%%%%%%%%%%%%%%%%%%%%%%%%%%%%%%%%%%%%%%%%%%%%%%%%%%%%%%%%%

\section{Basic example}\label{sec:example}

%%%%%%%%%%%%%%%%%%%%%%%%%%%%%%%%%%%%%%%%%%%%%%%%%%%%%%%%%%%%%%%%%%%%%%%%%%%

In $l_\infty^2$ any halfspace is a hyperconvex subspace, where its boundary line is isometric to $\mathbb{R}$. It is natural to glue two such halfspaces along its boundaries. Up to isometry a halfspaces is given by $H = \{(\xi_1,\xi_2) \in l_\infty^2 : \xi_2 \geq a \xi_1 \}$ for $0 \leq a \leq 1$. Observe that the gluing set $A = \{(\xi_1,\xi_2) \in l_\infty^2 : \xi_2 = a \xi_1 \}$ is externally hyperconvex in $H$ for $a=0$ and strongly convex in $H$ for $a=1$.

Now given two halfspaces we may assume that $H_1= \{(\xi_1,\xi_2) \in l_\infty^2 : \xi_2 \geq a \xi_1 \}$ and $H_2 =  \{(\xi_1,\xi_2) \in l_\infty^2 : \xi_2 \leq b \xi_1 \}$ with $0 \leq a \leq b \leq 1$. Then there are two possibilities to glue them depending on the orientation of the boundary line.

Let us first investigate the case where we glue by identifying $(\xi,a \xi)\in H_1$ with $(\xi,b \xi)\in H_2$. Then the resulting space is hyperconvex if and only if $a=b$.

For $a=b$ we obtain a space which is isometric to $l_\infty^2$ and therefore hyperconvex. Otherwise consider the pairwise intersecting closed balls of radius $1$ and centers $x_1=(0,1-a) \in H_1$, $x_2=(0,-b-1) \in H_2$ and $x_3=(2,b-1) \in H_2$. The intersection of $B(x_1,1)$ with $H_2$ is given by the union of all balls with center $y \in A\cap B(x_1,1)$ and radius $r_y = 1-d(x_1,y)$. A calculation shows that
\begin{align*}
	B(x_1,1) \cap H_2 = \left\{ (\xi_1,\xi_2) \in H_2 : -1 \leq \xi_1 \leq 1 \text{ and } \xi_2 \geq \frac{b-a}{1+a} \xi_1 - \frac{1+b}{1+a}a \right\}
\end{align*}  
and therefore the three balls have no common intersection point for $a < b$.

It remains to consider the case where we reflect $H_1$ before gluing, i.e. $H_1= \{(\xi_1,\xi_2) \in l_\infty^2 : \xi_2 \geq -a \xi_1 \}$ and $H_2 =  \{(\xi_1,\xi_2) \in l_\infty^2 : \xi_2 \leq b \xi_1 \}$ with $0 \leq a \leq b \leq 1$ glued by identifying $(\xi,-a \xi)\in H_1$ with $(\xi,b \xi)\in H_2$.
If $a=b=1$ or $a=b=0$ we have again that the resulting space is injective according to Theorems~\ref{thm:strongly_convex} and \ref{thm:externally_hyperconvex}. Indeed in both cases it is isometric to $l_\infty^2$. Otherwise choose again $x_1=(0,1-a) \in H_1$ as above. Then we get
\begin{align*}
	B(x_1,1)\cap H_2 = \left\{ (\xi_1,\xi_2) \in H_2 : -1 \leq \xi_1 \leq 1 \text{ and } \xi_2 \geq \max \left\{ l , m \xi_1- q \right\} \right\}.
\end{align*}
for $l=-1 + (1-b) \cdot \frac{1-a}{1+a} \leq -b, m=\frac{a+b}{1-a}$ and $q=a \cdot \frac{1+b}{1-a}$. Especially we have $(-1,l) \in B(x_1,1)\cap H_2$ and $(1,\xi_2) \in B(x_1,1)\cap H_2$ if and only if $\xi_2 = b$. Therefore if $a\neq 1$ and $b \neq 0$ the three balls with radius $1$ and centers $x_1=(0,1-a) \in H_1$, $x_2=(0,l-1) \in H_2$ and $x_3=(2,b-1) \in H_2$ are pairwise intersecting but have no common intersection point. Hence the resulting space is not hyperconvex.

\begin{figure}[ht!]
	\centering
	
	\labellist
	\small\hair 2pt
	\pinlabel $x_1$ [rb] at 88 245
	\pinlabel $x_2$ [rt] at 88 50
	\pinlabel $x_3$ [lt] at 180 95
	\endlabellist	
	\includegraphics[scale=0.5, width=84.12mm, height=121.3mm]{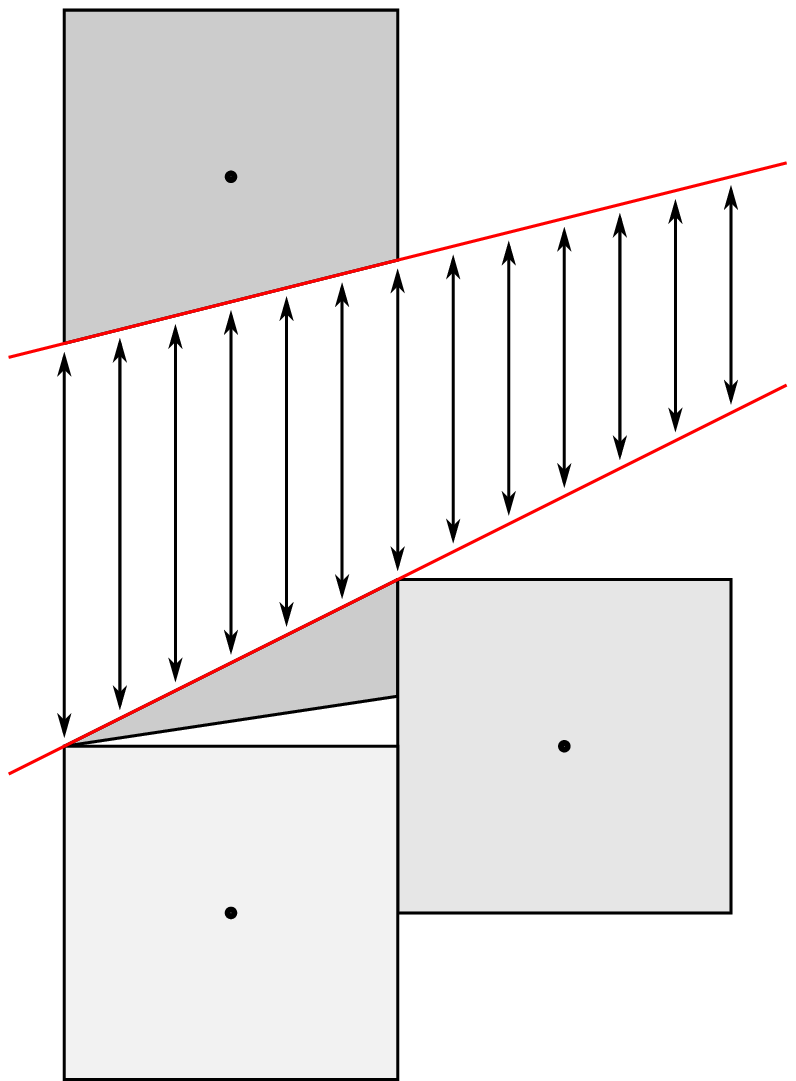}
	\labellist
	\small\hair 2pt
	\pinlabel $x_1$ [rb] at 88 295
	\pinlabel $x_2$ [rt] at 88 45
	\pinlabel $x_3$ [lt] at 180 105
	\endlabellist	
	\includegraphics[scale=0.5, width=84.12mm, height=121.36mm]{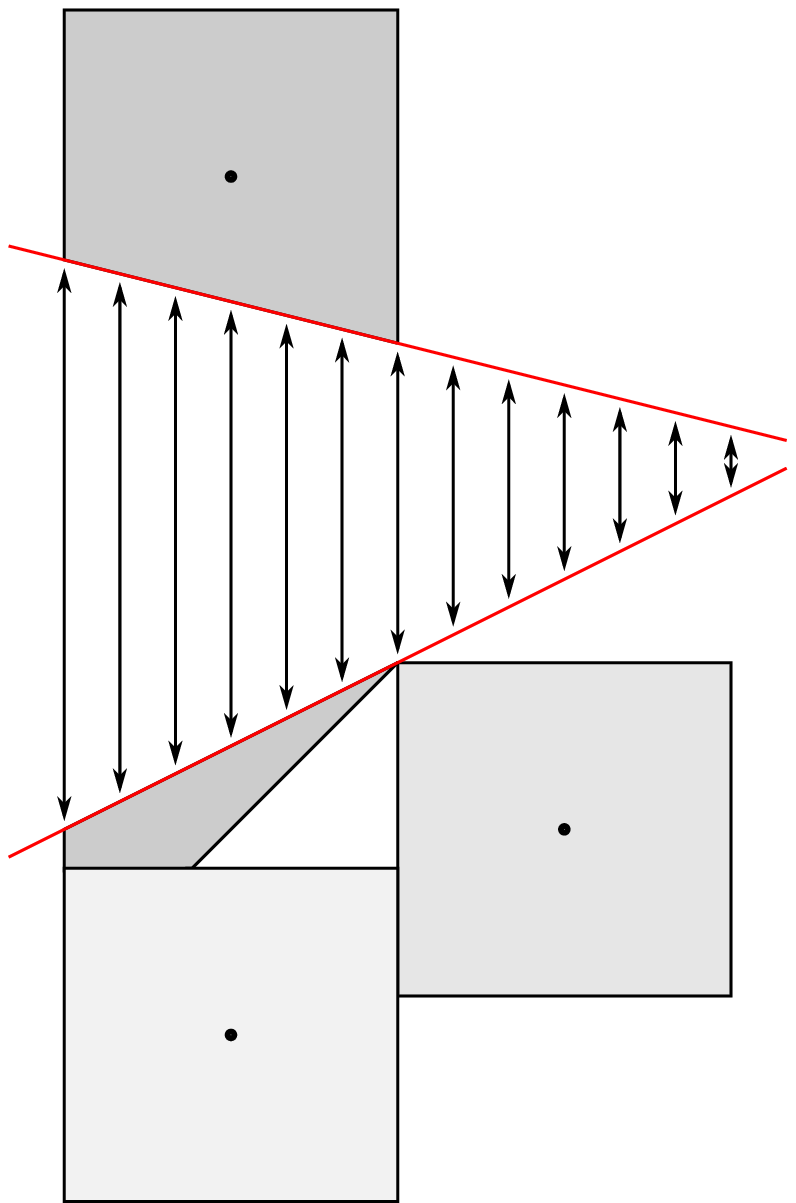}
	\caption{Gluing halfspaces in $l_\infty^2$.}
	\label{fig:gluing_example}
	
\end{figure}

%\begin{figure}[ht!]
%	\centering
%	
%	\labellist
%	\small\hair 2pt
%	\pinlabel $p_1$ [rt] at 88 38
%	\pinlabel $p_2$ [lt] at 170 80
%	\pinlabel $p_3$ [rb] at 75 135
%	\endlabellist	
%	\includegraphics[scale=.7]{pictures/gluing_example_3}
%	\labellist
%	\small\hair 2pt
%	\pinlabel $p_1$ [rt] at 88 38
%	\pinlabel $p_2$ [lt] at 170 80
%	\pinlabel $p_3$ [rb] at 70 125
%	\endlabellist
%	\includegraphics[scale=.7]{pictures/gluing_example_4}
%	\caption{Gluing halfspaces with $0 < b < a$ and $0 < -b < a$.}
%	\label{fig:gluing_example}
%\end{figure}

%\textbf{Acknowledgments.} I would like to thank Prof. Dr. Urs Lang for reading this and earlier versions of the paper and making helpful suggestions. This work was partially supported by the Swiss National Science Foundation.
	
%%%%%%%%%%%%%%%%%%%%%%%%%%%%%%%%%%%%%%%%%%%%%%%%%%%%%%%%%%%%%%%%%%%%%%%%%%%
% Bibliographie
%%%%%%%%%%%%%%%%%%%%%%%%%%%%%%%%%%%%%%%%%%%%%%%%%%%%%%%%%%%%%%%%%%%%%%%%%%%
\newpage

%\nocite{*}

\bibliographystyle{amsplain}
\bibliography{mybib}

%%%%%%%%%%%%%%%%%%%%%%%%%%%%%%%%%%%%%%%%%%%%%%%%%%%%%%%%%%%%%%%%%%%%%%%%%%%

\end{document}